\documentclass[12pt,a4paper]{article}
\usepackage{amssymb,amsmath,amsthm}
\usepackage{graphicx}

\newcommand\bC{{\mathbf C}}
\newcommand\cA{{\mathcal A}}

\newcommand\cC{{\mathcal C}}
\newcommand\cD{{\mathcal D}}

\newcommand\cF{{\mathcal F}}
\newcommand\cG{{\mathcal G}}
\newcommand\cH{{\mathcal H}}

\theoremstyle{plain}
\newtheorem{theorem}{Theorem}[section]
\newtheorem{lemma}[theorem]{Lemma}

\newtheorem{conjecture}[theorem]{Conjecture}
\newtheorem{proposition}[theorem]{Proposition}

\theoremstyle{definition}

\newtheorem{claim}[theorem]{Claim}

\newcommand\lref[1]{Lemma~\ref{lem:#1}}
\newcommand\tref[1]{Theorem~\ref{thm:#1}}
\newcommand\cref[1]{Corollary~\ref{cor:#1}}
\newcommand\clref[1]{Claim~\ref{clm:#1}}

\newcommand\cjref[1]{Conjecture~\ref{conj:#1}}

\newcommand\pref[1]{Proposition~\ref{prop:#1}}

\title{Forbidden subposet problems for traces of set families}
\author{D\'aniel Gerbner\thanks{Research supported by the J\'anos Bolyai Research Fellowship of the Hungarian Academy of Sciences and the National Research, Development and Innovation Office -- NKFIH under the grant the grant K 116769.}, Bal\'azs Patk\'os\thanks{Research supported by the J\'anos Bolyai Research Fellowship of the Hungarian Academy of Sciences and the National Research, Development and Innovation Office -- NKFIH under the grants SNN 116095 and K 116769.}, M\'at\'e Vizer\thanks{Research supported by the National Research, Development and Innovation Office -- NKFIH under the grants SNN 116095.}\\
\small Alfr\'ed R\'enyi Institute of Mathematics, Hungarian Academy of Sciences\\
\small P.O.B. 127, Budapest H-1364, Hungary.\\
\smallskip
\small \texttt{gerbner,patkos@renyi.hu, vizermate@gmail.com}\\
\smallskip }

\begin{document}

\maketitle

\begin{abstract}
In this paper we introduce a problem that bridges forbidden subposet and forbidden subconfiguration problems. 
The sets $F_1,F_2, \dots,F_{|P|}$ form a \textit{copy} of a poset $P$, if there exists a bijection $i:P\rightarrow \{F_1,F_2, \dots,F_{|P|}\}$ such that for any $p,p'\in P$ the relation $p<_P p'$ implies $i(p)\subsetneq i(p')$. A family $\cF$ of sets is \textit{$P$-free} if it does not contain any copy of $P$. The trace of a family $\cF$ on a set $X$ is $\cF|_X:=\{F\cap X: F\in \cF\}$.

We introduce the following notions: $\cF\subseteq 2^{[n]}$ is \textit{$l$-trace $P$-free} if for any $l$-subset $L\subseteq [n]$, the family $\cF|_L$ is $P$-free and $\cF$ is \textit{trace $P$-free} if it is $l$-trace $P$-free for all $l\le n$. As the first instances of these problems we determine the maximum size of trace $B$-free families, where $B$ is the butterfly poset on four elements $a,b,c,d$ with $a,b<c,d$ and determine the asymptotics of  the maximum size of $(n-i)$-trace $K_{r,s}$-free families for $i=1,2$. We also propose a generalization of the main conjecture of the area of forbidden subposet problems.
\end{abstract}


\vspace{2mm}

\noindent
{\bf AMS Subj.\ Class.\ (2010)}: 06A07, 05D05

\section{Introduction}

In this paper we introduce a problem that bridges two areas of extremal finite set theory, namely forbidden subposet problems and traces of set families. We denote by $[n]$ the set of the first $n$ positive integers and for a set $X$ we use the notation $2^{X}, \binom{X}{k},\binom{X}{\le k}, \binom{X}{\ge k}$ to denote the family of all subsets of $X$, all subsets of $X$ of size $k$ (that we also call $k$-subsets of $X$), all subsets  of $X$ of size at most $k$, and all subsets  of $X$ of size at least $k$, respectively. The family $\binom{X}{k}$ is often called the \textit{$k$th level of $X$}. 
Throughout the paper we use standard order notions.

\vskip 0.2truecm

\noindent
We will use multiple times the following well-known fact: $$\left| \binom{[n]}{\le \lfloor n/2- n^{2/3}\rfloor}\cup \binom{[n]}{\ge \lfloor n/2+ n^{2/3}\rfloor} \right|= o \left(\frac{1}{n}\binom{n}{\lfloor n/2\rfloor} \right).$$ 

Using this, we will assume several times throughout the paper that all members of a family $\cF$ have cardinalities in the interval $[n/2-n^{2/3},n/2+n^{2/3}]$ (as this way we lose only $o (\frac{1}{n}\binom{n}{\lfloor n/2\rfloor})$ sets). Note that for our purposes it is always going to be enough to use the interval $[n/3,2n/3]$ instead of $[n/2-n^{2/3},n/2+n^{2/3}]$.

\vskip 0.4truecm

\textbf{Forbidden subposet problems}. The very first result in extremal finite set theory is due to Sperner \cite{S}, who proved that if a family $\cF\subseteq 2^{[n]}$ does not contain two sets in inclusion, then the size of $\cF$ is at most $\binom{n}{\lfloor n/2\rfloor}$ and the only families achieving this size are $\binom{[n]}{\lfloor n/2\rfloor}$ and $\binom{[n]}{\lceil n/2\rceil}$. This was later generalized by Erd\H os \cite{Er}, who showed that if $\cF \subseteq 2^{[n]}$ does not contain a \textit{chain} of length $k+1$ (i.e. nested sets $F_1\subsetneq F_2 \subsetneq \dots \subsetneq F_{k+1}$), then the size of $\cF$ is at most $\sum_{i=1}^k\binom{n}{\lfloor \frac{n-k}{2}\rfloor +i}$, the sum of the $k$ largest binomial coefficients of order $n$. There is a vast literature of Sperner type problems (see the not very recent monograph of Engel \cite{En}), we focus on forbidden subposet problems introduced by Katona and Tarj\'an \cite{KT}. We say that the sets $F_1,F_2, \dots,F_{|P|}$ form a \textit{copy} of a poset $P$, if there exists a bijection $i:P\rightarrow \{F_1,F_2, \dots,F_{|P|}\}$ such that for any $p,p'\in P$ the relation $p<_P p'$ implies $i(p)\subsetneq i(p')$. A family $\cF$ of sets is \textit{$P$-free} if it does not contain any copy of $P$. Katona and Tarj\'an initiated the study of the parameter $La(n,P)$, the maximum size of a $P$-free family $\cF\subseteq 2^{[n]}$. Note that with this notation the above-mentioned result of Erd\H os can be formulated as $$La(n,P_{k+1})=\sum_{i=1}^k\binom{n}{\lfloor \frac{n-k}{2}\rfloor +i},$$ where $P_{k+1}$ denotes the chain of size $k+1$. As a copy of a chain of length $|P|$ in a family $\cF$ is always a copy of $P$, the result of Erd\H os implies $$La(n,P)\le (|P|-1)\binom{n}{\lfloor n/2\rfloor}.$$ Therefore it is natural to ask for the existence and value of $$\lim_{n \rightarrow \infty}\frac{La(n,P)}{\binom{n}{\lfloor n/2\rfloor}},$$ denoted by $\pi(P)$. It is not known whether $\pi(P)$ exists for every poset, but the precise or asymptotic value of $La(n,P)$ has been determined for many posets and in all known cases the (asymptotically) optimal construction consists of some of the middle levels of $[n]$. This motivated the following  conjecture that first appeared in \cite{GL}.

\begin{conjecture}
\label{conj:bigpos} For any poset $P$ let $e(P)$ denote the largest integer $k$ such that for any $j$ and $n$ the family $\cup_{i=1}^k\binom{[n]}{j+i}$ is $P$-free. Then $\pi(P)$ exists and is equal to $e(P)$. 
\end{conjecture}

\cjref{bigpos} was proved for many classes of posets. Let us state one of the nicest results of the area. To do so we need the following definition.
For a poset $P$ the Hasse diagram, denoted by $H(P)$, is a graph whose vertices are elements of $P$, and $xy$ is an edge if $x<y$ and there is no other element $z$ of $P$ with $x<z<y$. A poset $T$ is a \textit{tree poset} if its Hasse diagram is a tree. Let $h(T)$ denote the length of a longest chain in $T$. Bukh proved the following.

\begin{theorem}[\cite{B}]\label{thm:bukh}
For any tree poset $T$, we have $$La(n,T)=(h(T)-1+o(1))\binom{n}{\lfloor n/2\rfloor}.$$
\end{theorem}

\vskip 0.4truecm

\textbf{Traces of set families}. The trace of a set family is its restriction to a subset of its underlying set. Formally, the \textit{trace of a set $F$ on another set $X$} is $$F|_X:=F\cap X,$$ and the trace of a family $\cF$ on $X$ is $$\cF|_X:=\{F|_X:F\in \cF\}.$$ As different sets can have the same trace, we obtain $|\cF|_X|\le |\cF|$.
The fundamental result about traces of set families is the so-called Sauer-lemma proved independently by Sauer \cite{Sa}, Shelah \cite{Sh}, and Vapnik and Chervonenkis \cite{VC}. 

\begin{theorem}\label{thm:sauer}
If the size of a family $\cF\subseteq 2^{[n]}$ is larger than $\sum_{i=0}^{k-1}\binom{n}{i}$, then there exists a $k$-subset $X$ of $[n]$ such that $\cF|_X=2^X$ holds.
\end{theorem}

The bound in \tref{sauer} is tight as shown by $\binom{[n]}{\le k-1}$ and $\binom{[n]}{\ge n-k+1}$, but there are many other extremal families and a complete characterization is not yet known. This theorem leads naturally in several directions. One of them is the area of \textit{forbidden subconfigurations}. If $\cH\subseteq 2^{[k]}$ is a fixed family, then one can ask for the maximum size of a 'big' family $\cF\subseteq 2^{[n]}$ such that for any $k$-subset $X$ of $[n]$, the trace $\cF|_X$ does not contain a subfamily isomorphic to $\cH$. For more details, the interested reader is referred to the survey of Anstee \cite{A} and the references within. Naturally, one can consider several forbidden configurations at once. To mix the areas of forbidden subposet problems and forbidden subconfigurations, we will forbid all configurations that can be described by a poset structure.

\vspace{2mm}

We say that $\cF\subseteq 2^{[n]}$ is \textit{$l$-trace $P$-free} if for any $l$-subset $L\subseteq [n]$, the family $\cF|_L$ is $P$-free. A family $\cF$ is \textit{trace $P$-free} if it is $l$-trace $P$-free for all $l\le n$. Let $Tr(n,P)$ be the maximum size of a trace $P$-free family $\cF\subseteq 2^{[n]}$ and $Tr_l(n,P)$ be the maximum size of an $l$-trace $P$-free family $\cF\subseteq 2^{[n]}$.

\vspace{2mm}

If the traces of two sets on some set $X$ are in inclusion, then they remain in inclusion for any subset $Y$ of $X$, however the traces might coincide on $Y$. So it is not straightforward from definition that forbidding a subposet in the trace on a smaller subset is a stronger property than doing the same on a larger subset. However, we will prove the following rather easy monotonicity result.

\begin{proposition}\label{prop:monotone}
For a poset $P$ let $E(P)$ denote the number of edges in the Hasse diagram $H(P)$. If $E(P)\le k \le l$, then we have $$Tr_k(n,P)\le Tr_l(n,P).$$
\end{proposition}

\pref{monotone} implies that for any integer $k$ we have $Tr(n,P_{k+1})=Tr_k(n,P_{k+1})$. The value $Tr_k(n,P_{k+1})=\sum_{i=0}^{k-1}\binom{n}{i}$ follows from \tref{sauer}. The second author proved in \cite{P1} that the only $k$-trace $P_{k+1}$-free families are $\binom{[n]}{\le k-1}$ and $\binom{[n]}{\ge n-k+1}$ (moreover, he showed that for any fixed $k\le l$ we have $Tr_l(n,P_{k+1})=\sum_{i=0}^{k-1}\binom{n}{i}$ if $n$ is large enough and the only extremal families are $\binom{[n]}{\le k-1}$ and $\binom{[n]}{\ge n-k+1}$). 


Note that for any poset $P$ if $y(P)$ denotes the largest integer $m$ with $2^{[m]}$ not containing a copy of $P$, then by \tref{sauer} we have $Tr(n,P)\le \sum_{i=0}^{y(P)}\binom{n}{i}$, i.e. $Tr(n,P)$ grows polynomially in $n$.

The simplest non-chain posets are $\bigvee$ and $\bigwedge$, both being a poset on 3 elements $a,b,c$ with $a<_{\bigvee} b,c$ and $a,b<_{\bigwedge} c$. As they are both subposets of $P_3$, we have $Tr(n,\bigvee), Tr(n,\bigwedge)\le n+1$ and taking complements yields $Tr(n,\bigvee)=Tr(n,\bigwedge)$. Moreover, we know that there exist trace $\bigwedge$-free and trace $\bigvee$-free families of size $n+1$, namely $\binom{[n]}{\le 1}$ and $\binom{[n]}{\ge n-1}$.  The first contains $\bigvee$ and does not contain $\bigwedge$, while it is the opposite for the second family. On the other hand if we forbid both $\bigvee$ and $\bigwedge$ as traces, then the family cannot have more than 2 sets.

As a first non-trivial and non-chain instance of the problem of finding $Tr(n,P)$ we will consider the butterfly poset $B$ on 4 elements $a,b,c,d$ with $a,b<_B c,d$.

\begin{theorem}\label{thm:B}
For $n\ge 4$ we have $$Tr(n,B)=\lfloor 3n/2\rfloor+1.$$
\end{theorem}

As remarked above, \tref{sauer} implies that $Tr(n,P)$ grows polynomially in $n$ and the same argument shows that for any fixed $l> y(P)$ we have $Tr_l(n,P)=O(n^{l-1})$. The situation completely changes when $l$ is close to $n$. By definition, we have $Tr_n(n,P)=La(n,P)$. Observe that if $n$ is large enough and $\cF$ consists of consecutive levels of $[n]$, say $\cF=\cup_{i=j}^{j'}\binom{[n]}{i}$, then for any $(n-k)$-subset $X$ of $[n]$ we have $\cF|_X=\cup _{i=j-k}^{j'}\binom{X}{i}$. In particular, if $j'-j+k+1\le e(P)$, then $\cF|_X$ is $P$-free. This shows the inequality $$Tr_{n-k}(n,P)\ge (e(P)-k+o(1))\binom{n}{\lfloor n/2\rfloor}.$$ Therefore we propose the following generalization of \cjref{bigpos}.

 \begin{conjecture}\label{conj:ltrace}
 For any poset $P$ and integer $k<e(P)$ we have $$Tr_{n-k}(n,P)=(e(P)-k+o(1))\binom{n}{\lfloor n/2\rfloor}.$$

\noindent
Moreover, if $k\ge e(P)$, then $Tr_{n-k}(n,P)=o(\binom{n}{\lfloor n/2\rfloor})$ holds.
 \end{conjecture}
 
 Note that to see the moreover part of \cjref{ltrace}, by \pref{monotone}, it is enough to prove its statement for $k=e(P)$.
 
 \cjref{ltrace} was verified for chains by the second author in \cite{P2} and he obtained the exact value of $Tr_{n-1}(n,P_{k+1})$ for any positive integer $k$ in \cite{P3}. We prove \cjref{ltrace} for the posets $K_{r,s}$ on $r+s$ elements $a_1,a_2,\dots,a_r,b_1,b_2,\dots, b_s$ with $a_i< b_j$ for any $1\le i\le r$ and $1\le j \le s$. We will use the notation $\bigwedge_r$ for $K_{r,1}$ and $\bigvee_s$ for $K_{1,s}$. Note that $e(K_{r,s})=2$ if $r$ and $s$ are both at least two and $e(\bigvee_s)=e(\bigwedge_r)=1$. \cjref{bigpos} was verified for $K_{r,s}$ by De Bonis and Katona \cite{DK}. Therefore the following theorem implies \cjref{ltrace} in the case of the posets $K_{r,s}$.
 
 \begin{theorem}\label{thm:2level}
For any  positive integer $s\ge 1$, we have 

\vspace{3mm}

\textbf{(i)} $$\frac{s}{n}\binom{n}{\lfloor n/2 \rfloor}\le Tr_{n-1}(n,\bigvee\nolimits_s)\le \left(\frac{3s^3}{n}+o\left(\frac{1}{n}\right)\right)\binom{n}{\lfloor n/2 \rfloor}.$$

\vspace{2mm}

Furthermore, if $r,s\ge 2$, then we have

\vspace{3mm}
\textbf{(ii)} $Tr_{n-1}(n,K_{r,s})=(1+o(1))\binom{n}{\lfloor n/2\rfloor}$, and

\vspace{2mm}

\textbf{(iii)} $Tr_{n-2}(n,K_{r,s})\le \frac{6((s+1)^2+(r+1)^2)}{n}\binom{n}{\lfloor n/2\rfloor}$.
\end{theorem}

The smallest poset for which \cjref{bigpos} has not yet been proved is the diamond poset $D$ on four elements $a,b,c,d$ with $a< b,c < d$. The best known upper bound on $La(n,D)$ is due to Gr\'osz, Methuku, and Tompkins \cite{GMT}. We will prove that the moreover part of \cjref{ltrace} holds for $D$.

\begin{theorem}
\label{thm:diamond} For the diamond poset $D$ we have $$Tr_{n-2}(n,D)=O\left(\frac{1}{n^{1/3}}\binom{n}{\lfloor n/2\rfloor}\right).$$
\end{theorem}

The remainder of this paper is organized as follows: Section 2 deals with trace $P$-free families, \tref{B} along with some further remarks are shown there. A general result on $(n-1)$-traces of families that implies \tref{2level} is proved in Section 3 along with \tref{diamond} and other results about $l$-trace $P$-free families. Finally, Section 4 contains some concluding remarks.

\section{Trace $P$-free families}

\tref{sauer} has many proofs in the literature. One of them (obtained independently by Alon \cite{Al} and Frankl \cite{F}) uses down-compression. For a set $F$ and an element $i$, the \textit{down-compression operator} is defined as $$D_i(F):=F\setminus \{i\},$$ and for a family of sets $\cF$ we define $$D_i(\cF):=\{D_i(F): F\in \cF, \ D_i(F)\notin \cF\}\cup \{F:F, D_i(F)\in\cF\}.$$ It was proved in \cite{Al,F} that if we are given a family $\cF\subseteq 2^{[n]}$ such that there does not exist a $k$-set $X$  with $\cF|_X=2^X$, then the same holds for $D_i(\cF)$ for any $i \in [n]$. As any family $\cF$ can be turned into a \textit{downward closed} family (a family $\cD$ for which $C\subset D\in \cD$ implies $C\in \cD$) by applying a finite number of down-compressions, to prove \tref{sauer} it is enough to show its statement for downward closed families, which is rather straightforward. 

Observe that the trace $P$-free property is not preserved by down-compression, however there is a way how to obtain bounds on $Tr(n,P)$ by considering only downward closed families. Frankl in \cite{F} introduced the arrow relation $(n,m)\rightarrow (k,l)$ which, by definition, holds if for any family $\cF\subseteq 2^{[n]}$ of size $m$, there exists a $k$-set $X$ such that $|\cF|_X|\ge l$. With this notation \tref{sauer} can be formulated as $$(n,1+\sum_{i=0}^{k-1}\binom{n}{i})\rightarrow (k,2^k)$$ for any pair $n\ge k$. Frankl used down-compression to prove the following.

\begin{theorem}[\cite{F}]\label{thm:down}
The following statements are equivalent.

\vspace{2mm}

(i) $(n,m)\rightarrow (k,l)$

\vspace{1mm}

(ii) For every downward closed family $\cD\subseteq 2^{[n]}$ of size $m$, there exists a $k$-set $X$ such 

that $|\cD|_X|\ge l$.
\end{theorem}

We want to make use of \tref{down} to determine $Tr(n,P)$. In order to do that we make two simple observations. First note that if for some $k$-set $X$ the trace $\cF|_X$ contains more than $La(k,P)$ sets, then $\cF$ cannot be trace $P$-free. Therefore we obtain the following.

\begin{proposition}\label{prop:prop1}
For every poset $P$ we have $$Tr(n,P)\le\min \{m:\exists k\, (n,m)\rightarrow (k,La(k,P)+1)\}-1.$$
\end{proposition}

One can go one step further and improve \pref{prop1}. Suppose one determined the value of $Tr(k,P)$ for some small integer $k$. Then obviously, if for some $k$-set $X$ the trace $\cF|_X$ contains more than $Tr(k,P)$ sets, then $\cF$ cannot be trace $P$-free, so we obtain the following.

\begin{proposition}\label{prop:prop2}
For every poset $P$ we have $$Tr(n,P)\le\min \{m:\exists k\, (n,m)\rightarrow (k,Tr(k,P)+1)\}-1.$$
\end{proposition}



\vskip0.3truecm

Now we continue with the proof of \tref{B}.

\begin{lemma}\label{lem:dani0}
We have $$Tr(5,B)=8.$$
\end{lemma}

\begin{proof} We start with the following simple claim.
\begin{claim}\label{clm:Barrow} We have
$$(5,9)\rightarrow (3,6) \ \textrm{ and } \ (5,9)\rightarrow (4,7).$$
\end{claim}

\begin{proof}[Proof of Claim]
By \tref{down}, it is enough to prove the statement for downward closed families $\cD\subseteq 2^{[5]}$ of size 9. If $\cD$ contains a set $D$ of size 3, then $|2^{D}|=8\ge 6$. Otherwise $\cD$ contains at least 3 sets of size 2. As they are all subsets of $[5]$, for two of them $D_1,D_2$, we have $|D_1 \cup D_2|=3$ and as $\cD$ is downward closed, we have $|\cD\cap 2^{D_1\cup D_2}|\ge 6$.

Similarly we have either three 2-sets on three vertices or two 2-sets on four vertices, both cases give 7 sets on three or four vertices.
\end{proof}

Suppose $\cF\subseteq 2^{[5]}$ is a $B$-trace free family of size 9. Then by \clref{Barrow} there exists a 3-set $X$ with $|\cF|_X|\ge 6$. We may suppose that $X=[3]$ and as $\cF$ is $B$-trace free, we must have $$\cF|_{[3]}=\binom{[3]}{1}\cup \binom{[3]}{2}.$$

\begin{claim}\label{clm:dani}
Suppose there is a set $F\in \cF$ with $4\in F$. Then we have either 

\vspace{2mm}

$\bullet$ $\cF|_{[4]}=\binom{[4]}{2}$, or 

\vspace{1mm}

$\bullet$ $\cF|_{[4]}$ is isomorphic to $\{\{2\}, \{3\}, \{1,4\}, \{2,3\}, \{1,2,4\}, \{1,3,4\}  \}$, or 

\vspace{1mm}

$\bullet$ $\cF|_{[4]}$ is isomorphic to $\{\{1,4\}, \{2,4\}, \{3,4\}, \{1,2,4\}, \{1,3,4\}, \{2,3,4\}  \}$.

\end{claim}

\begin{proof}[Proof of Claim] The set $F$ intersects $[3]$ in a $1$ or $2$-element set. We separate cases according to this. We introduce the notation $\cA:=\cF|_{[4]}$, $\cA_i:=\cF|_{[4]\setminus \{i\}} \ (i \in [4])$. In particular, we have seen so far that $\cF|_{[3]}=\cA_4=\binom{[3]}{1}\cup \binom{[3]}{2}$.

\vspace{3mm}

\textbf{Case 1}. $\{1,4\}\in F|_{[4]}= \cA$.

\vspace{2mm}

\hspace{3mm}\textbf{Case 1.1}. $\{1,2,4\}\in \cA$. Let us consider $\cA_3$. We have $\{1,2,4\}$ and $\{1,4\}$ are in $\cA_3$. Also as we have $\{3\}\in \cA_4$ we have either $\emptyset$ or $\{4\}$ is in $\cA_3$. Thus we cannot have $\{1\}\in\cA_3$, hence $\{1,3\}\not\in \cA$. As we have $\{1,3\}\in \cA_4$ we must have $\{1,3,4\}\in\cA$. 

Also only one of $\{2\}$ or $\{2,4\}$ can be in $\cA$ as otherwise they would form a copy of $B$ in $\cA_3$ with $\{1,2,4\}$ and $\emptyset$ or $\{4\}$. 

\vspace{2mm}

\hspace{5mm}\textbf{Case 1.1.1}. $\{2\}\in \cA$. In this case $\{2,3,4\}\not\in \cA$, otherwise $\cA_1$ would contain $\{2\}$, $\{4\}$, $\{2,4\}$ and $\{2,3,4\}$. As $\{2,3\}\in \cA_4$, we must have $\{2,3\}\in\cA$. Thus we know $\{1,4\}$, $\{1,2,4\}$, $\{1,3,4\}$, $\{2\}$, $\{2,3\}$ are all in $\cA$. If $\{3,4\}$ was in $\cA$, then $\cA_2$ would contain $\{1,3,4\}, \{3,4\}, \{3\}$ and $\emptyset$, a contradiction. As $\{3\}\in \cA_4$, we must have $\{3\}\in\cA$. It is easy to see that no other set can be added in this case.

\vspace{2mm}

\hspace{5mm}\textbf{Case 1.1.2}. $\{2\}\not\in \cA$.
As $\{2\}\in \cA_4$, we must have $\{2,4\}\in \cA$. 

\vspace{2mm}

\hspace{8mm}\textbf{Case 1.1.2.1}.
$\{2,3\} \in \cA$. Then $\{3,4\}$ cannot be in $\cA$, as that would give $\{3\}$, $\{4\}$, $\{3,4\}$, $\{1,3,4\}$ in $\cA_2$. As $\{3\}\in \cA_4$, we have $\{3\}\in \cA$, but then $\cA_3$ contains $\emptyset$, $\{2\}$, $\{2,4\}$ and $\{1,2,4\}$, a contradiction. 

\vspace{2mm}

\hspace{8mm}\textbf{Case 1.1.2.2}. $\{2,3\} \not\in \cA$. As $\{2,3\} \in \cA_4$, we have that $\{2,3,4\} \in \cA$. If $\{3\}$ is in $\cA$, then $\cA_2$ contains $\{3\}$, $\{4\}$, $\{3,4\}$ and $\{1,3,4\}$, a contradiction. So $\{3\}\not\in \cA$, but $\{3\}\in \cA_4$, hence we must have $\{3,4\}\in\cA$. Thus $\{1,2,4\}$, $\{1,4\}$, $\{1,3,4\}$, $\{2,4\}$, $\{2,3,4\}$, $\{3,4\}$ are in $\cA$. Note that every additional set would create a copy of $B$ in $\cA_1$ except for $\{2,3\}$, $\{1,2,3\}$ and $\{1,2,3,4\}$. However, in this case $\{2,3\}$ is not in $\cA$ and neither $\{1,2,3\}$ nor $\{1,2,3,4\}$ can be in $\cA$ because $\{1,2,3\}\not\in\cA_4$.

\vspace{4mm}

\hspace{3mm}\textbf{Case 1.2}. $\{1,2,4\}\not\in \cA$. By symmetry we can also assume $\{1,3,4\}\not\in \cA$, otherwise we go back to Case 1.1. Hence we have $\{1,2\}, \{1,3\}\in \cA$. Then $\{1\},\{1,2\},\{1,4\}\in \cA_3$, thus $\emptyset$ cannot be in $\cA_3$, hence $\{3\}\not\in \cA$, thus $\{3,4\}\in\cA$, and similarly $\{2,4\}\in\cA$. Let us consider $\cA_1$ now. It contains $\{3\},\{3,4\},\{4\}$, thus it cannot contain $\{2,3,4\}$, hence $\{2,3,4\}\not\in\cA$, thus $\{2,3\}\in\cA$, i.e $\cA$ contains $\binom{[4]}{2}$. It is easy to see that no other set can be added.

\vspace{3mm}

\textbf{Case 2}. There are no 2-element sets in $\cA$ that contain 4. Then $\{1\},\{2\},\{3\}\in\cA$. We may assume $\{1,2,4\}=F|_{[4]}\in\cA$. Let us consider $\cA_1$. It contains $\emptyset,\{2\},\{2,4\}$ by the above. Also as $\{2,3\}\in \cA_4$, we have either $\{2,3\}$ or $\{2,3,4\}$ in $\cA$, and any of these complete a copy of $B$.

\vspace{2mm}

We are done with the proof of \clref{dani}.
\end{proof}

We are now ready to prove \lref{dani0}. Notice that in all cases of \clref{dani}, we have $|\cF|_{[4]}|=6$. We will show that $|\cF|_Y|\le 6$ holds for every other 4-element subset $Y$ of $[5]$ as well, which contradicts $(5,9)\rightarrow (4,7)$. 

Let us consider the possible outcomes of \clref{dani}. Let $Z=Y\setminus \{5\}$, then we have either $\cF|_Z=\binom{Z}{1}\cup \binom{Z}{2}$ or $\cF|_Z$ is a copy of the diamond poset. In the first case we can apply \clref{dani} this time $[3]$ replaced by $Y\cap [4]$ and $4$ by $5$ to obtain $|\cF|_Y| \le 6$. In the second case notice that $\cF|_Y\subseteq \cF|_Z\cup \{F\cup \{5\}:F\in \cF|_Z\}$. As this latter family is a copy of $2^{[3]}$, to ensure the $B$-free property, we must have $|\cF|_Y|\le 6$.


\end{proof}

\begin{lemma}\label{lem:B2} If $n\ge 6$, then we have $$(n,\lfloor 3n/2\rfloor +2) \rightarrow (5,9).$$
\end{lemma}

\begin{proof}
It is enough to verify the statement for downward closed families $\cD \subseteq 2^{[n]}$ of size $\lfloor 3n/2\rfloor+2$. If $\cD$ contains a set $D$ of size 3, then there exists $x\notin D$ with $\{x\}\in \cD$, and thus $|\cD|_{D\cup \{x\}}|\ge 9$. So we may assume $\cD \subseteq \binom{[n]}{\le 2}$. If $\cD$ does not contain two 2-sets with non-empty intersection, then $|\cD\cap \binom{[n]}{2}|\le \lfloor n/2\rfloor$ and we are done. If $D_1,D_2\in\cD$ are 2-sets with non-empty intersection and $D_3\in \cD\cap \binom{[n]}{2}$ is disjoint from $D_1\cup D_2$, then $\cD|_{D_1\cup D_2\cup D_3}\supseteq 2^{D_1}\cup 2^{D_2}\cup 2^{D_3}$ and we are done. 

This mean that $\cD\cap \binom{[n]}{2}$ is either a triangle or a star. In the former case we have $|\cD|\le 3+n+1<\lfloor 3n/2\rfloor +2$. In the latter case, if the star consists of at most 3 sets, then again  $|\cD|\le 3+n+1<\lfloor 3n/2\rfloor +2$, while if the star consists of at least 4 sets $D_1,D_2,D_3,D_4$, then $|\cD|_{D_1\cup D_2 \cup D_3 \cup D_4}|= 10$.
\end{proof}

Now the upper bound in \tref{B} follows from \pref{prop2}, \lref{dani0} and \lref{B2}. For the lower bound we consider a family that consists of the empty set, all the $1$-element sets, and $\lfloor n/2 \rfloor$ pairwise disjoint $2$-element sets. It is easy to see that this family does not contain the butterfly poset, and as it is downward closed, it does not contain it as a trace either. This finishes the proof of \tref{B}. $\hfill\qed$
\vskip 0.2truecm
We state the last observation as a general lower bound. Let

\vskip 0.2truecm
 $$La_D(n,P):=\max\{|\cF|\subseteq 2^{[n]}: \cF\, \textrm{is $P$-free and downward closed}\}, \ \textrm{and} $$ $$La_U(n,P):=\max\{|\cF|\subseteq 2^{[n]}: \cF\, \textrm{is $P$-free and upward closed}\}.$$ 

\vspace{3mm}

\begin{proposition}\label{prop3} We have $$Tr(n,P)\ge \max\{La_D(n,P),La_U(n,P)\}.$$

\end{proposition}

Let $x(n,P)$ be the largest integer such that $\binom{[n]}{\le x(n,P)}$ does not contain $P$. It is easy to see that $x(n,P)$ is monotone decreasing in $n$, so we can define its limit $x(P)$ and $x(n,P)=x(P)$ for $n$ large enough. It is easy to see that $$La_D(n,P)\ge \sum_{i=0}^{x(n,P)}\binom{n}{i}\ge \sum_{i=0}^{x(P)}\binom{n}{i}.$$ 
Remember that $y(P)$ is the largest integer such that $2^{[y(P)]}$ does not contain $P$. If the size of a family $\cF\subseteq 2^{[n]}$ is larger than $\sum_{i=0}^{y(P)}\binom{n}{i}$, then by \tref{sauer} it contains a subset $X$ of size $y(P)+1$ such that $\cF|_X=2^X$ holds. Obviously $2^X$ contains a copy of $P$ by the definition of $y(P)$, thus we have $Tr(n,P)\le\sum_{i=0}^{y(P)}\binom{n}{i}$. By the observations above we have $$\sum_{i=0}^{x(P)}\binom{n}{i}\le La_D(n,P) \le Tr(n,P) \le \sum_{i=0}^{y(P)}\binom{n}{i}.$$

If for a poset $P$ we have $x(P)=y(P)$, then $Tr(n,P)=\sum_{i=0}^{x(P)}\binom{n}{i}$. In particular, 

\begin{proposition} If a poset $P$ has a unique maximum element, then $$Tr(n,P)=\sum_{i=0}^{x(P)}\binom{n}{i}.$$

\end{proposition}

\section{$l$-traces}

We start this section by showing the monotonicity of $Tr_l(n,P)$ in $l$.

\begin{proof}[Proof of \pref{monotone}.]
Let $E(P)\le k \le l$ and suppose $\cF\subseteq 2^{[n]}$ is $k$-trace $P$-free. We claim that $\cF$ is $l$-trace $P$-free. Assume otherwise. Then there exists an $l$-subset $X$ of $[n]$ such that $\cF|_X$ contains a copy $F_1|_X,F_2|_X,\dots, F_{|P|}|_X$ of $P$. For every edge $e$ of the Hasse diagram $H(P)$ let $x_e$ be an element from $F_i|_X\setminus F_j|_X$ if $F_i|_X$ and $F_j|_X$ are the sets corresponding to the end vertices of $e$. As $E(P)\le k$, we obtain that $|\{x_e:e\in E(H(P))\}|\le k$. Therefore, for any $k$-set $Y$ containing $\{x_e:e\in E(H(P))\}$ we have that $F_1|_Y,F_2|_Y,\dots,F_{|P|}|_Y$ form a copy of $P$ in $\cF|_Y$ contradicting the $k$-trace $P$-free property of $\cF$.
\end{proof}

We continue with the proof of \tref{2level}.

 \begin{proof}[Proof of \tref{2level} (i)]
 To see the lower bound consider the following well-known construction. Let us partition $$\binom{[n]}{\lfloor n/2\rfloor}=\cF_1\cup \cF_2 \cup \dots \cup \cF_n$$ such that $$\cF_i:=\left\{F\in \binom{[n]}{\lfloor n/2\rfloor}: \sum_{j\in F}j \equiv i \ (\text{mod}\ n) \right\}.$$ Let $\cF$ be the union of the $s$ largest $\cF_i$'s, and therefore we have $$|\cF|\ge \frac{s}{n}\binom{[n]}{\lfloor n/2\rfloor}.$$ For any $x \in [n]$, the trace $\cF_{[n]\setminus \{x\}}$ contains sets of size $\lfloor n/2\rfloor$ and $\lfloor n/2\rfloor-1$, so a copy of $\bigvee_s$ would be possible if there existed an $(\lfloor n/2\rfloor-1)$-set $G$ that is contained in at least $s+1$ sets of $\cF$. By construction, there is no such $G$, therefore $\cF$ is $(n-1)$-trace $\bigvee_s$-free.
 
To prove the upper bound let $\cF\subseteq 2^{[n]}$ be an $(n-1)$-trace $\bigvee_s$-free family. Let $\cF_1=\cF\cap \binom{[n]}{\le n/2+n^{2/3}}\cap \binom{[n]}{\ge n/2-n^{2/3}}$. Note that $\cF_1$ cannot contain a chain of length $s+1$ as omitting an element of its smallest set would result in a $(s+1)$-chain in the trace contradicting the $(n-1)$-trace $\bigvee_s$-free property. Therefore $\cF_1$ contains an antichain $\cF_2$ with $$|\cF_2|\ge |\cF_1|/s.$$ We will bound the size of $\cF_2$ using the Lubell-function $$\lambda(\cF_2)=\sum_{F\in \cF_2}\frac{1}{\binom{n}{|F|}}.$$ To this end we will count the number of pairs $(F,\cC)$ with $\cC$ being a maximal chain in $[n]$ and $F\in \cF_2\cap \cC$. We will denote by $\bC$  the set of all maximal chains in $[n]$.
 
 Let us consider $\cG$, the \textit{shadow} of $\cF_2$, $$\cG:=\{G:\exists x\in F\in \cF_2: G=F\setminus \{x\}\}$$ and for a set $G\in  \cG$ let $$\bC_G:=\{\cC\in \bC: G\in \cC\}.$$
 
 \begin{claim}\label{clm:dama}
 For every chain $\cC\in \bC$ there exist at most $s$ sets $G\in \cG$ with $\cC\in \bC_G$.
 \end{claim}
 
 \begin{proof}[Proof of claim.]
 Suppose to the contrary that $\cC\in \bC$ contains $G_1\subsetneq G_2\subsetneq \dots \subsetneq G_{s+1}$ with $G_i\in \cG$ for all $1\le i \le s+1$. Then there exist $x_1,x_2,\dots, x_{s+1}$ with $x_i\notin G_i$ and $F_i:=G_i\cup \{x_i\}\in \cF_2$. But then we have $$F_i|_{[n]\setminus \{x_1\}}\supsetneq F_1|_{[n]\setminus \{x_1\}}=G_1$$ for $i=2,3,\dots,s+1$ and they are all different, as $\cF_2$ is an antichain, thus these form a copy of $\bigvee_s$.
 \end{proof}
 
 \begin{claim}\label{clm:bubi}
 For every $G\in \cG$ there exist at most $s$ sets $F\in \cF_2$ with $G\subseteq F$.
 \end{claim}
 
  \begin{proof}[Proof of claim.]
 As $G\in \cG$ there exists an $x\notin G$ with $F=G\cup \{x\}\in\cF_2$. As $\cF_2$ is an antichain, any other $F'\in \cF_2$ with $G\subseteq F'$ must not contain $x$. So if there were $F_1,F_2,\dots, F_s\in \cF_2$ other than $F$ all containing $G$, then the traces of $F,F_1,F_2,\dots, F_s\in \cF_2$ on $[n]\setminus \{x\}$ would form a copy of $\bigvee_s$.
 \end{proof}
Let us now count the number of pairs $(F,\cC)$ with $\cC$ being a maximal chain in $[n]$ and $F\in \cF_2\cap \cC$. On the one hand it is $$\sum_{F\in \cF_2}|F|!(n-|F|)!.$$ On the other hand it is at most $$\sum_{G\in \cG}\sum_{\cC\in \bC_G}\sum_{F\in \cC\cap \cF_2}1.$$ As $\cF_2$ is an antichain, no chain $\cC\in \bC_G$ can contain a set $F\in \cF_2$ with $F\subseteq G$. Therefore, by \clref{bubi} and the condition that $\cF_2$ consists only of sets of size from $[n/2-n^{2/3},n/2+n^{2/3}]$, we have $$\sum_{\cC\in \bC_G}\sum_{F\in \cC\cap \cF_2}1\le \frac{s}{n-|G|}|\bC_G|\le \frac{3s}{n}|\bC_G|.$$ \clref{dama} yields $$\sum_{G\in \cG}|\bC_G|\le s|\bC|=s\cdot n!,$$ and thus we obtain
 \[
 \sum_{F\in \cF_2}|F|!(n-|F|)!\le \frac{3s^2}{n}n!.
 \] 
 Dividing by $n!$ gives
 \[
 \lambda(\cF_2)\le \frac{3s^2}{n},
 \]
 and thus $$|\cF_2|\le \frac{3s^2}{n}\binom{n}{\lfloor n/2\rfloor},$$ which implies $$|\cF_1|\le \frac{3s^3}{n}\binom{n}{\lfloor n/2\rfloor}.$$ As $|\binom{[n]}{\le n/2-n^{2/3}} \cup \binom{[n]}{\le n/2+n^{2/3}}| = o(\frac{1}{n}\binom{n}{\lfloor n/2\rfloor})$, the proof of \tref{2level} (i) is finished.
 \end{proof}

\begin{proof}[Proof of \tref{2level} \textbf{(iii)}]
 The statement follows from part (i) and the following claim. We denote by $K_{r,1,s}$ the poset on $r+1+s$ elements $a_1,a_2,\dots,a_r,c,b_1,b_2,\dots, b_s$ with $a_i< c< b_j$ for any $1\le i\le r$ and $1\le j \le s$. 
 
\begin{claim}
 For any pair $r,s\ge 2$ of positive integers, the inequality $$Tr_{n-2}(n,K_{r,1,s})\le Tr_{n-1}(n,\bigvee\nolimits_{2s+1})+Tr_{n-1}(n,\bigwedge\nolimits_{2r+1})$$ holds.
\end{claim}
 
 \begin{proof}[Proof of claim]
Let $\cF \subseteq 2^{[n]}$ be a family of size $Tr_{n-1}(n\bigvee_{2s+1})+Tr_{n-1}(n,\bigwedge_{2r+1})+1$. We can find pairs $(F_i,x_i)$ for $1\le i \le Tr_{n-1}(n,\bigwedge_{2r+1})+1$ and $F_i \in \cF$, $x_i \in [n]$ such that all $F_i$'s are distinct and $F_i|_{[n]\setminus x_i}$ is the bottom element of a copy of $\bigvee_{2s+1}$ in $\cF|_{[n]\setminus \{x_i\}}$. Therefore there exists a $y\in[n]$ such that $$\{F_i|_{[n]\setminus \{y\}}:1\le i \le Tr_{n-1}(n,\bigwedge\nolimits_{2r+1})+1\}$$ contains a copy of $\bigwedge_{2r+1}$, say $$F_1|_{[n]\setminus \{y\}}, F_2|_{[n]\setminus \{y\}}, \dots, F_{2r+2}|_{[n]\setminus \{y\}}$$ with $F_1|_{[n]\setminus \{y\}}$ being the top element. We claim that $\cF|_{[n]\setminus \{x_1,y\}}$ contains a copy of $K_{r,1,s}$. Indeed, let $$F_1, G_1,G_2,\dots, G_{2s+1}\in \cF$$ be sets the traces of which form a copy of $\bigvee_{2s+1}$ on $[n]\setminus \{x_1\}$ (these sets exist by the definition of $F_1$ and $x_1$). As removing one element may cause at most 2 sets to have the same trace, $F_1$ and at least $s$ of the $G_i$'s will have distinct traces on  $[n]\setminus \{x_1,y\}$ and thus will form a copy of $\bigvee_r$ with $F_1|_{[n]\setminus \{x_1,y\}}$ being the bottom element. The same reasoning shows that we can pick $r$ of $F_2,F_3,\dots, F_{2r+2}$ such that their traces on $[n]\setminus \{x_1,y\}$ together with $F_1|_{[n]\setminus\{x_1,y\}}$ form a copy of $\bigwedge_r$ with $F_1|_{[n]\setminus\{x_1,y\}}$ being the top element. Putting these copies of $\bigvee_s$ and $\bigwedge_r$ together, we obtain a copy of $K_{r,1,s}$.
\end{proof}
Note that $K_{r,s}$ is a subposet of $K_{r,1,s}$, hence \tref{2level} \textbf{(iii)} is proved.
 \end{proof}

 Let $T$ be a tree poset with a unique maximum element $m$. We define two new posets obtained from $T$. Let $T^k$ denote the poset obtained from $T$ by replacing $m$ with an antichain of size $k$. 
Equivalently, $$T^k=T\setminus \{m\}\cup \{m_1,m_2,\dots,m_k\}$$ such that the $m_i$'s form an antichain, for any $t\in T\setminus \{m\}$ and $1\le i\le k$ we have $t<_{T^k}m_i$ and for any $t,t'\in T\setminus \{m\}$ we have $t<_{T^k} t'$ if and only if $t<_T t'$. Note that $T^k$ is not a tree poset unless there is a unique element of $T$ that precedes $m$. Also, if $T^k$ is not a tree poset, then $e(T^k)=e(T)+1=h(T)$.
 
Let $T^{\otimes r}$ be the tree poset defined recursively (with respect to its height) in the following way: if $T=P_1$ is the poset with one element, then $T^{\otimes r}=P_1$ for any $r$. Otherwise, if the maximum element $m$ of $T$ has $c$ children in its Hasse-diagram and the posets below its children are $T_1,T_2,\dots, T_c$, then the maximum element of $T^{\otimes r}$ has $c\cdot r$ children $m_1,m_2,\dots, m_{cr}$ such that $m_{(j-1)r+i}$ is the maximum element of a poset isomorphic to $T^{\otimes r}_j$ for every $1\le j\le c$ and $1\le i \le r$.

 \begin{theorem}\label{thm:tree}
 For any integer $s$ and tree poset $T$ with a unique maximum element we have $$Tr_{n-1}(n,T^s)=(e(T)+o(1))\binom{n}{\lfloor n/2\rfloor}.$$
 \end{theorem}
 
 \begin{proof} The proof relies on the following lemma.
 \begin{lemma}\label{lem:relation}
 For any integer $s$ we have $$Tr_{n-1}(n,T^s)\le La(n, T^{\otimes 2})+Tr_{n-1}(n,\bigvee\nolimits_{s-1})+1.$$
 \end{lemma}
 
 \begin{proof}[Proof of lemma]
 Let $\cF\subseteq 2^{[n]}$ be a family of size $La(n, T^{\otimes 2})+Tr_{n-1}(n,\bigvee_{s-1})+1$. Then $\cF$ contains a copy of $T^{\otimes 2}$. Let $F_1$ be the set of this copy corresponding to the top element of $T^{\otimes 2}$. As $\cF\setminus \{F_1\}$ is still larger than $La(n,T^{\otimes 2})$, we can pick a set $F_2\in \cF\setminus \{F_1\}$ that corresponds to the top element of $T^{\otimes 2}$ in a copy in $\cF\setminus \{F_1\}$. Repeating this, we can obtain sets $F_1,F_2,\dots, F_{Tr_{n-1}(n,\bigvee_{s-1})+1}$ with the property that for every $F_j$ there exists a copy of $T^{\otimes 2}$ in $\cF$ in which they correspond to the top element. Let us write $$\cF':=\{F_1,F_2,\dots, F_{Tr_{n-1}(n,\bigvee_{s-1})+1}\}.$$ By definition, there exists $x\in [n]$ such that $\cF'|_x$ contains a copy of $\bigvee_{s-1}$, say $F_1\setminus \{x\},F_2\setminus \{x\}, \dots, F_s\setminus \{x\}$. We claim that $\cF|_{[n]\setminus \{x\}}$ contains a copy of $T^s$ with $F_1\setminus \{x\},F_2\setminus \{x\}, \dots, F_s\setminus \{x\}$ playing the role of the $s$ top elements of $T^s$. 
 
Indeed, without loss of generality we can assume that $F_s\setminus \{x\}\subsetneq F_i\setminus \{x\}$ holds for all $1\le i \le s-1$. We know that there exists a copy of $T^{\otimes 2}$ in $\cF$ with $F_s$ playing the role of the top of $T^{\otimes 2}$. We claim that we can take some of the sets (including $F_s$) of this copy of $T^{\otimes 2}$ such that their traces on $[n]\setminus \{x\}$ form a copy of $T$ and thus together with  $F_1\setminus \{x\},F_2\setminus \{x\}, \dots, F_{s-1}\setminus \{x\}$ they form a copy of $T^s$ in $\cF_{[n]\setminus \{x\}}$. To see this, we only need to observe that if $G_1,G_2\subsetneq G$ and $G_1\neq G_2$, then for any $y$ at least one of $G_1\setminus \{y\},G_2\setminus \{y\}$ is a strict subset of $G\setminus \{y\}$. So we can pick the sets of the copy of $T$ recursively starting with $F_s$.

Thus we indeed obtained a copy of $T^s$ in $\cF|_{[n]\setminus \{x\}}$.
 \end{proof}
 Now the upper bound in \tref{tree} follows from \lref{relation}, \tref{2level} \textbf{(i)} and \tref{bukh} using the simple observation that the height of $T$ and $T^{\otimes r}$ are the same and therefore we have $e(T)=e(T^{\otimes r})$ for any integer $r$.
 
 The lower bound is due to the general observation made before \cjref{ltrace} that $Tr_{n-1}(n,P)\ge (e(P)-1-o(1))\binom{n}{\lfloor n/2\rfloor}$ holds for any poset $P$.
 \end{proof}

Note that \tref{2level} \textbf{(ii)} follows by applying \tref{tree} to $T=\bigwedge_r$.

\vskip 0.3truecm
In the remainder of this Section, we prove \tref{diamond}. We will use the following lemma.

\begin{lemma}\label{lem:317}
Let $G$ be a graph on $n$ vertices and let $\ell:E\rightarrow \mathbb{R}$ be a labeling of the edges such that in any 4-cycle the edges with the smallest and largest $\ell$-value are adjacent (if there are more edges with smallest or largest $\ell$-value, then all these pairs of edges are adjacent). Then $G$ cannot contain a complete bipartite graph with partite sets of size 3 and 17.
\end{lemma}

\begin{proof}[Proof of lemma]
We can assume that $\ell$ is injective as that makes the weakest restriction. Suppose towards a contradiction that $G$ contains 20 vertices $A,B,C$ and $v_1,v_2,\dots,v_{17}$ such that $A,B,C$ are connected to all $v_i$'s. By rearranging, we may assume that $\ell(Av_i)<\ell(Av_j)$ whenever $i<j$. By the famous result of Erd\H os and Szekeres \cite{ESz}, there exist five vertices $v_{i_1},\dots,v_{i_5}$ ($i_1<i_2<i_3<i_4<i_5$) such that the sequence $l(Bv_{i_j}$) $j=1,2,3,4,5$ is monotone. Applying again the Erd\H os -Szekeres result we find three veritces $\alpha,\beta,\gamma$ among the $v_{i_j}$'s such that
$$\ell(A\alpha), \ell(A\beta), \ell(A\gamma);$$ $$\ell(B\alpha), \ell(B\beta), \ell(B\gamma);$$ $$\ell(C\alpha), \ell(C\beta), \ell(C\gamma)$$ all form monotone sequences. So two of these triples are monotone decreasing or increasing. By rearranging if necessary, we may suppose that $$\ell(A\alpha)<\ell(A\beta)<\ell(A\gamma); \ \ \ell(B\alpha)<\ell(B\beta)<\ell(B\gamma); \textrm{ and } \ell(A\alpha)<\ell(B\alpha)$$ hold.

As $A\alpha$ is the smallest labeled edge in the cycles $A\alpha B\beta$ and $A\alpha B\gamma$, using that the smallest and the largest labeled edges must be adjacent, we obtain $l(B\beta)<l(A\beta)$ and $l(B\gamma)<l(A\gamma)$. But then in the cycle $A\beta B\gamma$ we have $l(B\beta)<l(A\beta), l(B\gamma)<l(A\gamma)$, so the smallest labeled edge is $B\beta$ and the largest labeled edge is $A\gamma$, contradicting that these should be adjacent.
\end{proof}

\begin{proof}[Proof of \tref{diamond}]
Let $\cF\subseteq 2^{[n]}$ be an $(n-2)$-trace diamond-free family. As $| \binom{[n]}{\le \lfloor n/2- n^{2/3}\rfloor}\cup \binom{[n]}{\ge \lfloor n/2+ n^{2/3}\rfloor}|= o(\frac{1}{n}\binom{n}{\lfloor n/2\rfloor})$, we may and will assume that all sets of $\cF$ have size from $[n/2-n^{2/3},n/2+n^{2/3}]$.

Let us consider a (symmetric) chain partition $\mathbb{C}$ of $2^{[n]}$, i.e. $\mathbb{C}$ consists of $\binom{n}{\lfloor n/2\rfloor}$ chains $\cC$ such that $\cup_{\cC\in \mathbb{C}}\cC=2^{[n]}$ and for any pair $\cC,\cC'\in \mathbb{C}$ we have $\cC\cap \cC'=\emptyset$. For any $\cC\in \mathbb{C}$ let us define the graph $G_\cC$ with vertex set $[n]$ and edge set $$\{e\in \binom{[n]}{2}:\exists C\in \cC \hskip .4truecm C \cup e \in \cF\}.$$ Let $e_\cC$ denote the number of edges in $\cG_\cC$ and let us bound $\sum_{\cC\in \mathbb{C}}e_\cC$. 

Every $F\in \cF$ contains $\binom{|F|}{2}$ pairs and each of them belongs to different chains. Moreover, for every $\cC$ and every edge $e\in E(G_{\cC})$ there can be at most 3 sets $F\in \cF$ containing $e$ and $F \setminus e\in \cC$ (as otherwise these sets would form a 4-chain, i.e. a special copy of the diamond), so we obtain 
$$\frac{1}{54}n^2|\cF|\le\frac{1}{3}\sum_{F\in \cF}\binom{|F|}{2}\le \sum_{\cC\in \mathbb{C}}e_\cC.$$

On the other hand for any $\cC\in\mathbb{C}$ let us define the labeling $\ell: E(G_\cC)\rightarrow \{0,1,\dots,n\}$ by letting $\ell(e):=|C|$ with $C\in\cC, C\cup e\in \cF$ (if there are more such sets $C$, then take the size of an arbitrary one). Note that $G_\cC$ and the labeling $\ell$ satisfy the conditions of \lref{317}. Indeed, if $e_1,e_2,e_3,e_4$ are consecutive edges of a 4-cycle in $G_\cC$ with $C_1,C_2,C_3,C_4\in \cC$ and $e_i\cup C_i=F_i\in \cF$ such that $|C_1|\le |C_2|,|C_4|\le |C_3|$, then the traces of the $F_i$'s on $[n]\setminus e_1$ form a copy of the diamond poset. \lref{317} implies $G_{\cC}$ does not contain a complete bipartite graph with parts of size 3 and 17. Therefore the celebrated K\H ov\'ari - T. S\'os - Tur\'an theorem \cite{KST} implies $e_\cC=O(n^{2-1/3})$ for all $\cC\in \mathbb{C}$. Summing over $\mathbb{C}$ we obtain
\[
\frac{1}{54}n^2|\cF|=O\left(\binom{n}{\lfloor n/2\rfloor}n^{2-1/3}\right).
\]
Rearranging yields the theorem.
\end{proof}

\section{Concluding remarks}

We finish this article by posing some remarks and problems concerning our results.

\vspace{2mm}

$\bullet$ We conjecture the following about the butterfly poset:

\begin{conjecture} If $n\ge 5$, then
$Tr_{n-1}(n,B)=\binom{n}{\lfloor n/2\rfloor}$.
\end{conjecture}

$\bullet$ We introduced the functions $La_D(n,P)$ and $La_U(n,P)$ as lower bounds on $Tr(n,P)$. They seem to be interesting on their own right, and we are not aware of any earlier study on them.
Natural questions arise about the order of magnitude of $La_D(n,P)$. 

\vspace{2mm}

It is natural to ask if we can find an upper bound on $La_D(n,P)$ using $x(P)$. However, we show a poset $P_m$ for every $m$ such that $x(P_m)=1$ and $La_D(n,P)=\Omega(n^m)$.

Let $(P_m,<)$ consist of a minimal element $a$, $2^m+1$ elements $b_1,\dots,b_{2^m+1}$ with $a< b_i$ for $1\le i\le 2^m+1$ and $m':=\binom{2^m+1}{2}$ elements $c_1,\dots, c_{m'}$ such that for every two different $b_k,b_l$ there is exactly one $c_j$ with $b_k,b_l< c_j$. Observe that we have $x(P_m)=1$ as a family consisting of sets of size at most $2$ is $P_m$-free if and only if its $2$-element sets do not contain a copy of the complete graph $K_m$. On the other hand consider a partition of $[n]$ into $m$ sets $A_1, \dots, A_m$ of almost equal size. Consider the family $\cF$ of sets that intersect every $A_j$ in at most one element. It is obvious that $\cF$ is downward closed and has cardinality $\Omega(n^m)$. We will show it is $P_m$-free.

Suppose by contradiction that $\cF$ contains a copy of $P_m$. Let $\cF_1$ be the subfamily consisting of the sets that correspond to $b_1,\dots, b_{2^m+1}$. If two distinct element of $A_j$ are both contained in members of $\cF_1$, then they are both contained in a set corresponding to $c_k$ for some $k$, which is impossible. Thus $\cup \cF_1$ intersects every $A_j$ in at most one vertex, which implies $|\cup \cF_1| \le m$. Therefore we have $|\cF_1|\le 2^m$, a contradiction.

\vspace{2mm}

$\bullet$ Concerning the connection of $La_D(n,P)$, $La_U(n,P)$ and $Tr(n,P)$, the obvious question is the following: is Proposition \ref{prop3} sharp for $n$ large enough? We know that $6=Tr(3,B)>\max\{La_D(3,B),La_U(3,B)\}=5$, but also that $Tr(n,B)=\max\{La_D(n,B),La_U(n,B)\}$ if $n>4$. The sharpness of Proposition \ref{prop3} would mean that we could use down-compression in forbidden subposet problems for traces, similarly to \tref{down}.

Another possible improvement that would essentially be equivalent to using down- compressions is at \pref{prop2}. Can we replace $Tr$ by $\max\{La_D,La_U\}$ in \pref{prop2}? For the butterfly poset and $n=3$ these are different but $(n,m)\rightarrow (5,9)$ would give the same bound. On the other hand, note that \pref{prop2} is sharp for any poset $P$. Indeed, $Tr(n,P)\ge\min \{m:\exists k\, (n,m)\rightarrow (k,Tr(k,P)+1)\}-1$, as shown by $k=n$. The question is if we can chose a small $k$. More precisely, is there a constant $c(P)$ for every poset $P$ such that determining $Tr(c(P),P)$ and using \pref{prop2} is enough to find $Tr(n,P)$ for every $n$ (like $c(P)=5$ for the butterfly poset)?





\bibliographystyle{acm}
\bibliography{tracesubpo}
\end{document}